\documentclass[reqno,centertags,a4paper,11pt]{amsart}
\usepackage{amssymb}
\usepackage{amsthm}
\usepackage{eucal}
\usepackage{color, soul, xcolor}
\soulregister\cite{7}
\soulregister\ref{7}
\soulregister\pageref7
\usepackage[english]{babel}
\usepackage{graphicx}
\usepackage{subfig}
\usepackage{hyperref}

\newcommand{\R}{\mathbb R}

\newcommand{\T}{\mathbb T}

\def\ra{\right>}
\def\la{\left<}
\def\lm{\left\|}
\def\rm{\right\|}

\newcommand{\re}{\text{Re}}
\newcommand{\im}{\text{Im}}

\newcommand{\p}{\partial}
\newcommand{\rr}{\mathbb{R}}

\renewcommand{\p}{\partial}

\numberwithin{equation}{section}
\newtheorem{theorem}{Theorem}[section]
\newtheorem{proposition}[theorem]{Proposition}
\newtheorem{remark}[theorem]{Remark}
\newtheorem{lemma}[theorem]{Lemma}
\newtheorem{corollary}[theorem]{Corollary}

\begin{document}
\title[Decay of the radius of spatial analyticity]{Decay of the radius of spatial analyticity for the modified KdV equation and the nonlinear Schr\"odinger equation with third order dispersion  }
\author{Renata O. Figueira and Mahendra  Panthee}
\address{Department of Mathematics, University of Campinas\\13083-859 Campinas, SP, Brazil}
\email{renof@unicamp.br, mpanthee@unicamp.br}



\maketitle

\begin{abstract}
We consider the initial value problems (IVPs) for the modified Korteweg-de Vries (mKdV) equation 
\begin{equation*}
\label{mKdV}
\left\{\begin{array}{l}
\partial_t u+ \partial_x^3u+\mu u^2\p_xu =0,
\quad x\in\mathbb{R},\; t\in \rr , \\
u(x,0) = u_0(x),
\end{array}\right.
\end{equation*}
where $u$ is a real valued function and $\mu=\pm 1$, and  the cubic nonlinear Schr\"odinger equation with third order dispersion (tNLS equation in short)
\begin{equation*}
\label{t-NLS}
\left\{\begin{array}{l}
\partial_t v+i\alpha \partial_x^2v+\beta \partial_x^3v+i\gamma |v|^2v
=
0,
\quad x\in\mathbb{R},\; t\in\rr , \\
v(x,0)
=
v_0(x),
\end{array}\right.
\end{equation*}
where $\alpha, \beta$ and $\gamma$ are real constants and $v$ is a complex valued function. In both problems, the initial data $u_0$ and $v_0$ are analytic on $\rr$ and have uniform radius of analyticity $\sigma_0$ in the space variable.

We prove that the both IVPs are locally well-posed for such data by establishing an analytic version of the trilinear estimates, and showed that the radius of spatial analyticity of the solution remains the same $\sigma_0$ till some lifespan $0<T_0\le 1$.
We also consider the evolution of the radius of spatial analyticity $\sigma(t)$ when the local solution extends globally in time and  prove that for any time $T\ge T_0$ it is bounded from below 
by  $c T^{-\frac43}$,  for the mKdV equation  in the defocusing case ($\mu = -1$) and by $c T^{-(4+\varepsilon)}$, $\varepsilon>0$, for the tNLS equation. The result for the mKdV equation improves the one obtained in \cite{BGK} and, as far as we know, the result for the tNLS equation is the new one.

\end{abstract}

 
 \vspace{0.2cm}

\section{Introduction}

In this work we consider the initial value problems (IVPs) associated to two dispersive models with real analytic initial data. The first model we consider is the modified Korteweg-de Vries (mKdV) equation 
\begin{equation}
\label{mKdV-IVP}
\left\{\begin{array}{l}
\partial_t u+ \partial_x^3u+\mu u^2\p_xu
=
0,
\quad x\in\mathbb{R},\; t\in \rr , \\
u(x,0)
=
u_0(x),
\end{array}\right.
\end{equation}
where $u$ is a real valued function and $\mu=\pm 1$. The next model is the cubic nonlinear Schr\"odinger equation  with third order dispersion (tNLS equation in short)
 \begin{equation}
\label{tNLS-IVP}
\left\{\begin{array}{l}
\partial_t v+i\alpha \partial_x^2v+\beta \partial_x^3v+i\gamma |v|^2v
=
0,
\quad x\in\mathbb{R},\; t\in\rr , \\
v(x,0)
=
v_0(x),
\end{array}\right.
\end{equation}
where $\alpha, \beta$ and $\gamma$ are real constants and $v$ is a complex valued function.

The mKdV equation \eqref{mKdV-IVP} is a generalization of the famous KdV equation \cite{KdV} and is known as focusing for $\mu =1$ and defocusing for $\mu=-1$.  The mKdV equation appears in several physical contexts, for example, propagation of waves in plasma \cite{KEC}, dynamics of traffic flow \cite{LL}, fluid mechanics \cite{He} and nonlinear optics  \cite{LGM, LM}  are a few to mention.

The mKdV equation \eqref{mKdV-IVP} has attracted much attention from both applied and theoretical perspectives. Various methods have been used to construct solutions, see for example \cite{H-73}, \cite{T-72}, \cite{PG} and references therein.  It possesses infinite number of conserved quantities, is both Hamiltonian and completely integrable, and can be solved by using the inverse scattering technique \cite{T-72}.

 Among  infinite number of conserved quantities possesses by \eqref{mKdV-IVP}, we highlight the mass 
\begin{equation}
\label{mass-mkdv}
M(u)(t):=\int u^2(x,t)dx,
\end{equation}
and the energy
\begin{equation}
\label{ener-mkdv}
E(u)(t):=\int\Big[(\partial_xu(x,t))^2-\frac \mu6 (u(x,t))^4\Big]dx,
\end{equation}
that will be useful in this work.

The  well-posedness of the IVP \eqref{mKdV-IVP} with data in the $L^2$-based Sobolev spaces $H^s(\R)$,
$s \in \R$, has long been studied in the literature, see \cite{BS-75}, \cite{Kato}, \cite{KPV}, \cite{KPV-1} and references therein. The optimal local well-posedness result for given data in $H^s(\R)$  was obtained  by Kenig, Ponce
and Vega \cite{KPV}, who exploited dispersion through the use of local smoothing estimates, for $s\geq\frac14$. An alternative proof of this result in $H^{\frac14}(\R)$ was given by Tao \cite{Tao-2} using the Fourier transform restriction norm method. Using the conserved quantities \eqref{mass-mkdv} and \eqref{ener-mkdv} one can get the global well-posedness result in $H^s(\R)$ for $s\geq 1$. The global well-posedness result for the low regularity data, viz.,  $s >\frac14$, was established, using the {\em I-method}, by Colliander et al \cite{CKSTT}, and at the end-point $s = \frac14$, by Kishimoto \cite{K-09}.  For further improvement we refer to  \cite{ FL, MPV} and references therein.


The tNLS equation \eqref{tNLS-IVP}, also known as the extended nonlinear Schr\"odinger equation,  appears to describe several physical phenomena like the nonlinear pulse propagation in an optical fiber, nonlinear modulation of a capillary gravity wave on water, for more details we refer to \cite{Agr-07}, \cite{XC-04}, \cite{DT-21}, \cite{HK-81}, \cite{MT-18}, \cite{Oikawa-93},  \cite{Tsutsumi-18} and references therein. In some literature, this model is also known as the third order Lugiato-Lefever equation \cite{MT-17}. 

The tNLS equation can also be considered as a particular case of the higher order nonlinear Schr\"odinger  equation proposed by  Hasegawa and Kodama in \cite{HK} and \cite{K} 
\begin{equation*}\label{honse}
\p_t v +i\alpha \p_x^2 v +\beta \p_x^3 v +i\gamma |v|^2v+\delta |v|^2\p_x v +\epsilon v^2 \p_x \overline{v}
=0,
\end{equation*}
where $\gamma,\delta,\epsilon \in \mathbb{C}$ and $\alpha,\beta \in \rr$ are constants, and  $v = v(x, t)$ is  complex valued function.

As presented in \cite{L1} (see also \cite{XC-04}), the $L^2$-norm
\begin{equation}
\label{mass-tnls}
\tilde{M}(v)(t):=\int |v(x,t)|^2 dx,
\end{equation} 
and the following quantity 
\begin{equation}\label{cons-01}
\tilde{E}(v)(t):=\int v(x,t)\overline{\p_x v(x,t)}d x,
\end{equation}
are conserved by the flow of \eqref{tNLS-IVP}.

The well-posedness issues and  other properties of solutions of the IVP \eqref{tNLS-IVP}  posed on $\R$ or $\T$ have extensively been studied by several authors, see for example \cite{XC-04},   \cite{CP-22}, \cite{DT-21}, \cite{MT-17}, \cite{OTT-19} and references threrein. The optimal local well-posedness result for the IVP~\eqref{tNLS-IVP} with given data in $H^s(\R)$,  $s>-\frac14$ is obtained  in \cite{XC-04}. The author in \cite{XC-04} also proved that the crucial trilinear estimate used to obtain the local well-posedness result fails whenever $s<-\frac14$. In this sense, the local well-posedness result for $s>-\frac14$ is the best possible using this technique.  Quite recently, the authors in \cite{CP-24} implemented the {\em I-method} to construct an {\em almost conserved quantity} and used it to obtain a sharp global well-posedness result to  the IVP \eqref{tNLS-IVP} for given data in $H^s(\R)$, $s>-\frac14$.


We note that the best local well-posedness results for the IVPs  \eqref{mKdV-IVP} and \eqref{tNLS-IVP} with given data in  $H^s(\R)$ were respectively obtained in \cite{Tao-2} and \cite{XC-04} using the Fourier transform norm spaces introduced in \cite{B, B-93}, commonly known as Bourgain's spaces. Generally speaking, Bourgain's spaces $X^{s,b}$, $s, b\in \R$ are very suitable to get well-posedness results for low regularity Sobolev data. These spaces are defined with the norm given by
\begin{equation}\label{Xsb-norm}
\|w\|_{X^{s,b}}
=
\|\la \xi\ra^{s} \la \tau-\phi(\xi)\ra^{b} |\widehat{w}(\xi,\tau)|\|_{L_{\tau,\xi}^2},
\end{equation}
where $\la \xi\ra:= 1+|\xi|$,  $w$ is a generic function which could be $u$ or $v$  and $\phi$ is the phase function associated to the linear part of the equation. In our case, 
\begin{align}
\label{phase-f}
\phi(\xi)
=\begin{cases}
\xi^3, & \text{ for the mKdV equation, }\\
\alpha \xi^2 +\beta\xi^3, & \text{ for the tNLS equation. }
\end{cases}
\end{align} 

As mentioned in the beginning, the main interest of this work is in considering the IVPs \eqref{mKdV-IVP} and \eqref{tNLS-IVP} with real analytic initial data $u_0$. For this purpose we consider the initial data $u_0$ in the Gevrey class $G^{\sigma,s}(\rr)$, $\sigma>0$ and $s\in\rr$ defined as follows
$$G^{\sigma,s}(\mathbb{R})
:=
\left\{ f\in L^2(\mathbb{R});\;
\|f\|_{G^{\sigma,s}(\mathbb{R})}^2
=
\int \la\xi\ra^{2s}e^{2\sigma |\xi|}|\widehat{f}(\xi)|^2 d{\xi}
<
\infty\right\},
$$
where $\widehat{f}$ denotes the Fourier transform given by 
$$
\widehat{f}(\xi)
=
c\int e^{-ix\xi}f(x)d x.
$$
When is convenient we also use the notation $\mathcal{F}(f)$ to denote the Fourier transform of $f$. Also, we will use $c$ or $C$ to denote constants whose value may vary from one line to the next.

From the Paley-Wienner theorem we have that every function in $G^{\sigma, s}(\rr)$ is analytic in space variable and admits an holomorphic extension to a complex strip $S_\sigma=\{ x+iy;\; |y|<\sigma\}$. In this context, $\sigma$ is called the \textbf{uniform radius of analyticity}.

While considering the existence of the solution to the IVPs with initial data in the Gevrey class $G^{\sigma,s}(\rr)$, the following two questions arise naturally. Starting with the given data $u(0)\in G^{\sigma,s}$, can one guarantee the existence of the solution such that the regularity in the space variable is sustained at least for short time? When one extends the local solution globally in time, the radius of analyticity $\sigma(t)$ most possibly decreases. In this situation, can one find the lower bound for the radius of analyticity $\sigma(t)$ when $t\to\infty$? In the recent time these sort of questions have attracted attention of several mathematicians, see for example \cite{BFH-gB}, \cite{BFH},  \cite{BGK}, \cite{GK-1}, \cite{HW}, \cite{SS-18}, \cite{SS}, \cite{T} and references therein. 

At this point, we mention the works in  \cite{GK-1} and \cite{BGK} where the authors obtained the well-posedness results for the IVP associated to the generalized KdV equations with data in the Gevrey class $G^{\sigma,s}(\rr)$. In \cite{BGK}, the authors also obtained an algebraic lower bound for the evolution of the radius of analyticity which turns out to be  $CT^{-12}$ for the both KdV and mKdV equations. Recently,  Selberg and Silva \cite{SS} introduced a concept of almost conserved quantities and  obtained $cT^{-(\frac43+\epsilon)}$ as a lower bound  for the radius of analyticity for the KdV equation improving the result in \cite{BGK}. Quite recently, this lower bound  has further been improved to $cT^{-\frac14}$ in \cite{HW} using the {\em I-method}. The NLS equation has also been extensively studied with data in the Gevrey class, see  \cite{AKS}, \cite{T} and references contained there. We emphasize here the work of Tesfahun \cite{T}, where the author  proved  the existence of the global solution to the defocusing cubic NLS equation belonging to $C([-T,T], G^{\sigma(T),s}(\rr))$ for any $T>0$ as long as $\sigma(T)\ge cT^{-1}$. As in the KdV equation, the main ingredient to obtain this result is  an almost conserved quantity derived in the $H^1$-level.

To get motivation for the present work we would like to mention the following. When the coefficient $\beta =0$ in \eqref{tNLS-IVP} one obtains the well known classical NLS equation with cubic nonlinearity.  When the coefficient $\alpha=0$ in \eqref{tNLS-IVP} one gets the dispersive term of the famous complex mKdV equation but  without derivative on the  cubic nonlinearity.  As mentioned in the previous paragraph, one has well-posedness results for both the mKdV and the NLS equations with data in the Gevrey class as well as the lower bounds for the evolution of the radius of analyticity. So, it is natural to ask, what happens to the IVP  \eqref{tNLS-IVP} with data in the Gevrey class $G^{\sigma, s}(\R)$  when $\alpha\ne 0$ and $\beta\ne 0$? Is it possible to obtain a lower bound for the evolution of the radius of analyticity $\sigma(t)$, $t\to\infty$? The next natural question concerning the IVP \eqref{mKdV-IVP} is, can one improve the algebraic lower bound for the radius of analyticity  obtained in \cite{BGK}?

The main objective of this work is to provide affirmative  answers to the questions posed above. For this, we will use the analytic version of the Bourgain's space, so called Gevrey-Bourgain space related to the mKdV and tNLS equations. Given $\sigma \ge 0$ and  $s,b\in \rr$, the Gevrey-Bourgain space  $X^{\sigma,s,b}(\rr^2)$ are defined as the closure of the Schwartz space under the norm
\begin{equation}\label{G-B-norm}
\|w\|_{X^{\sigma,s,b}}
=
\| e^{\sigma |\xi|}\la \xi\ra^{s} \la \tau-\phi(\xi)\ra^{b} |\widehat{w}(\xi,\tau)|\|_{L_{\tau,\xi}^2},
\end{equation}
where $\la \xi\ra:= 1+|\xi|$ and $\phi$ is the phase function given by \eqref{phase-f}.

Also, for $T>0$ we denote the Gevrey-Bourgain space restricted in time by $X_T^{\sigma,s,b}(\rr^2)$ with norm given by
\begin{equation*}
\| w\|_{X_T^{\sigma,s,b}}
=
\inf\big\{ \|\tilde{w}\|_{X^{\sigma,s,b}};\;
w=\tilde{w}
\text{ on } \rr\times (-T,T)\big\}.
\end{equation*} 
For $\sigma=0$ we recover the classical Bourgain's space with norm given by \eqref{Xsb-norm}. In this case, we simply have  $X^{0,s,b}\equiv X^{s,b}$ and $X_T^{0, s,b}\equiv X_T^{s,b}$. We introduce the operator $e^{\sigma|D_x|}$ given by 
\begin{equation}
\label{op.A}
\widehat{e^{\sigma|D_x|} w}(\xi)
=
e^{\sigma |\xi|}\widehat{w}(\xi),
\end{equation}
so that, one has
\begin{equation}\label{ssb}
\|e^{\sigma |D_x|} w\|_{X^{s,b}}
=
\|w\|_{X^{\sigma, s, b}}.
\end{equation}
The relation \eqref{ssb} allows us to translate the results in the classical Bourgain's spaces to the analytic version of them.

To avoid any possible confusion we introduce the following notations to distinguish the Gevrey-Bourgain's space related to the mKdV and tNLS equations
\begin{equation*}
\begin{split}
 &Y^{\sigma, s,b}: = X^{\sigma, s, b}, \quad {\text{when}}\quad \phi(\xi) = \xi^3,\\
&Z^{\sigma, s,b}: = X^{\sigma, s, b}, \quad {\text{when}}\quad \phi(\xi) = \alpha\xi^2+\beta\xi^3,
\end{split}
\end{equation*}
and similarly for $Y^{\sigma, s,b}_T$, $Z^{\sigma, s,b}_T$ and when $\sigma=0$, $Y^{ s,b}$, $Z^{ s,b}$, $Y^{ s,b}_T$, $Z^{ s,b}_T$ as well. From now on, we will use the spaces $X^{\sigma, s, b}$ and $X^{ s, b}$ to state and prove the results that hold for any phase function $\phi$ given by \eqref{phase-f}. We will use the  spaces $Y^{\sigma, s, b}$ and $Y^{ s, b}$  to state and prove the results that hold only for the mKdV equation and the  spaces $Z^{\sigma, s, b}$ and $Z^{ s, b}$  to state and prove the results that hold only for the tNLS equation.

Now we are in position to state the main results of this work. Regarding the local well-posedness, we prove the following results.
\begin{theorem}\label{lwp-mKdV-thm}
Let $\sigma>0$ and $s\ge \frac14$. For each $u_0\in G^{\sigma,s}(\rr)$ there exists a time 
$T_0=T_0(\|u_0\|_{G^{\sigma,s}})>0$ such that the IVP \eqref{mKdV-IVP} admits a unique solution $u$ in $ C([-T_0,T_0] ; G^{\sigma,s}(\rr))\cap Y_{T_0}^{\sigma,s,b}$. Moreover, the data-to-solution map is locally Lipschitz. 
\end{theorem} 

\begin{theorem}\label{lwp-tNLS-thm}
Let $\sigma>0$ and $s>-\frac14$. For each $v_0\in G^{\sigma,s}(\rr)$ there exists a time
$T_0=T_0(\|v_0\|_{G^{\sigma,s}})>0$ such that the IVP \eqref{tNLS-IVP} admits a unique solution $v$ in $C([-T_0,T_0] ; G^{\sigma,s}(\rr))\cap Z_{T_0}^{\sigma,s,b}$. Moreover, the data-to-solution map is locally Lipschitz. 
\end{theorem} 

\begin{remark}	Note that the Gevrey spaces $G^{\sigma,s}$ enjoy the following inclusion
\begin{equation}
\label{Gds.emb}
G^{\sigma,s}(\mathbb{R})
\subset
G^{\sigma',s'}(\mathbb{R}),\;
 \text{for all }0<\sigma'<\sigma \text{ and } s,s'\in\mathbb{R}.
\end{equation}
In view of this inclusion, it is sufficient to prove the local well-posedness theory for indices $\sigma>0$ and $s=s_0$ fixed, which in turn would imply the same in $G^{\sigma, s}(\R)$ for all $\sigma>0$ and $s\in \rr$.
\end{remark}

In the following theorems we state the main results regarding the global solution and the evolution of the radius of analyticity.
\begin{theorem}\label{global-mKdV-thm}
Let $\sigma_0>0$, $s\geq \frac14$, $u_0\in G^{\sigma_0,s}(\R)$ and $u$ be the local solution to the IVP~\eqref{mKdV-IVP} in the defocusing case $(\mu = -1)$ given by Theorem \ref{lwp-mKdV-thm}. Then, for any $T\ge T_0$ the local solution $u$ extends globally in time  satisfying
$$
u\in C([-T,T]; G^{\sigma(T),s}),
\quad\text{with}\quad
\sigma(T)\ge\min\Big\{\sigma_0, cT^{-\frac43}\Big\},
$$
where $c$ is a positive constant depending on $s$, $\sigma_0$ and $\|u_0\|_{G^{\sigma_0, s}}$. 
\end{theorem}

\begin{theorem}\label{global-tNLS-thm}
Let $\sigma_0>0$, $s>-\frac14$, $v_0\in G^{\sigma_0,s}(\R)$ and $v$ be the local solution to the IVP~\eqref{tNLS-IVP} given by Theorem \ref{lwp-tNLS-thm}. Then, for any $T\ge T_0$ the local solution $u$ extends globally in time  satisfying
$$
v\in C([-T,T]; G^{\sigma(T),s}),
\quad\text{with}\quad
\sigma(T)\ge\min\Big\{\sigma_0, cT^{-(4+\varepsilon)}\Big\},
$$
where $\varepsilon>0$ is arbitrarily small and $c$ is a positive constant depending on $s$, $\sigma_0$ and $\|v_0\|_{G^{\sigma_0, s}}$. 
\end{theorem}

We emphasize that the result of Theorem \ref{global-mKdV-thm} significantly improves the earlier result in \cite{BGK} where the authors obtained $cT^{-12}$ as a lower bound for the radius of analyticity for the mKdV equation. As far as we know, for the tNLS equation the results of theorems \ref{lwp-tNLS-thm} and \ref{global-tNLS-thm} are the first ones in this direction.

To prove the global results and the lower bounds for the evolution of the radius of analyticity stated in theorems \ref{global-mKdV-thm} and \ref{global-tNLS-thm}, we derived  the almost conserved quantities (ACQ) in the $G^{\sigma,s}(\mathbb{R})$ spaces of the form (see \eqref{ACL-mKdV-bound} and \eqref{ACL-tNLS-bound} below)
\begin{equation}
\label{alm-0}
A_{\sigma}(t) \leq A_{\sigma}(0) + C \sigma^{\theta}A_{\sigma}(0)^2, \qquad \theta>0,
\end{equation}
where $A_{\sigma}(t)$ is defined appropriately taking in consideration the conserved quantities.

Once having the almost conserved quantity \eqref{alm-0} at hand, one can extend the local solution to the global in time and obtain the lower bound for the radius of analyticity following the scheme developed in \cite{SS}. As can be seen in this process, the higher the value of $\theta$ in \eqref{alm-0} better the lower bound for the radius of analyticity in the sense that it decays much slower as time advances.

\begin{remark}
As can be seen in theorems  \ref{global-mKdV-thm} and \ref{global-tNLS-thm}, the lower bound for the radius of analyticity for the mKdV equation is better than that for the tNLS equation. The main difference for this is the level of Sobolev regularity used to construct almost conserved quantity. The IVP \eqref{mKdV-IVP} associated to the mKdV equation is locally well-posed for $s\geq \frac14$, so we constructed ACQ at the $H^1$-level using auxiliary trilinear estimate at $H^{\frac14}$-level. The  IVP \eqref{tNLS-IVP} associated to the  tNLS equation is locally well-posed for $s>-\frac14$, so we constructed ACQ at $L^2$-level using auxiliary trilinear estimate at $H^{-\frac14+\epsilon}$-level. The difference in the level of regularities used for the ACQs and the level of auxiliary trilinear estimate provides the exponent $\theta$ in \eqref{alm-0} that in turn provides the decay rate of the radius of analyticity in the form $cT^{-\frac1{\theta}}$. So, one may naturally ask why not using $H^1$-level to obtain ACQ for the tNLS equation as well to get a better result? Unfortunately, as already noted in \cite{L1}, the conserved quantity given by \eqref{cons-01} is not sign definite and cannot be used for our purpose.

\end{remark}

This paper is organized as follows. In Section \ref{Sec-2} we record some preliminary estimates and derive the trilinear estimates  in the analytic Gevrey-Bourgain spaces. Section \ref{Sec-3} is devoted to furnish the local well-posedness results with real analytic data. In Section \ref{Sec-4} we introduce almost conserved quantities and find the associated decay estimates. Finally, in  Section \ref{Sec-5} we extend  the local solution globally in time and obtain an algebraic lower bound for the radius of analyticity as stated in theorems \ref{global-mKdV-thm} and \ref{global-tNLS-thm}.
%

%
%
%
%
%
%
\section{Preliminaries and trilinear estimates} \label{Sec-2}

In this section we will derive some trilinear estimates that play crucial role in the  proof of the well-posedness results.  We start with some classical linear estimates in the Gevrey-Bourgain's spaces, whose proof can be found in \cite{BFH-gB, BGK, GK-1, SS-18} for instance. Before stating these results let $\psi\in C_0^\infty((-2,2))$ be a cut-off function 
with $0\leq \psi\leq 1$, $\psi(t)=1$ on $[-1,1]$ and $\psi_T(t)=\psi\left(\frac{t}{T}\right)$. Also, let $W(t)$ be the unitary group given by $\widehat{W(t)\varphi}=e^{it\phi(\xi)}\widehat{\varphi}$ where $\phi$ is the phase function defined in \eqref{phase-f}.

\begin{lemma}
Let $\sigma\ge 0$, $s\in\rr$, $b>\frac12$ and $b-1<b'<0$. Then, for all $0<T\le 1$ there is a constant $c=c(s,b)>0$ such that
\begin{equation}
\label{psiT.w.u-est}
\|\psi(t)W(t)f(x)\|_{X^{\sigma,s,b}}
\le
c \|f\|_{G^{\sigma,s}},
\end{equation}
and
\begin{equation}
\label{psiT.int.w.u-est.0}
\Bigg\| \psi_T(t)\int\limits_0^t W(t-t')w(x,t')d t'\Bigg\|_{X^{\sigma,s,b}}
\le
cT^{1-(b-b')}\|w\|_{X^{\sigma,s,b'}}.
\end{equation}
\end{lemma}

As we will see in the sequel, the proof of theorems \ref{lwp-mKdV-thm} and \ref{lwp-tNLS-thm} heavily rely on the trilinear estimates in the Gevrey-Bourgain's spaces $X^{\sigma,s,b}(\rr^2)$. 
We start with the following trilinear estimate in the classical Bourgain's spaces associated to the mKdV equation proved in \cite{Tao-2} (see Corollary $6.3$ there).
\begin{lemma}
\label{Corollary.6.3}
For all $u_1, u_2$, $u_3$ and $0<\varepsilon\ll 1$, we have
\begin{equation}
\label{trilinear}
\| \partial_x(u_1 u_2 u_3)\|_{Y^{\frac14,-\frac12+\varepsilon}}
\le
C\|u_1\|_{Y^{\frac14,\frac12+\varepsilon}}\|u_2\|_{Y^{\frac14,\frac12+\varepsilon}}\|u_3\|_{Y^{\frac14,\frac12+\varepsilon}},
\end{equation} 
with the constant $C>0$ depending only on $\varepsilon$.
\end{lemma}

Also considering the classical Bourgain's spaces, but now associated to the tNLS equation, the following  estimate was proved in \cite{XC-04} (see Lemma $2.2$ there).

\begin{lemma}
\label{Lemma2.2}
Let $- \frac14< s \le 0$, $b>\frac7{12}$ and $b'<\frac{s}3$. Denoting $\eta=(\xi,\tau)$, $\eta_1=(\xi_1,\tau_1)$ and $\eta_2=(\xi_2,\tau_2)$, consider
\begin{equation}\label{Knn}
K(\eta,\eta_1,\eta_2)
=
\frac{\la\xi\ra^{s}\la\xi+\xi_1-\xi_2\ra^{-s}\la\xi_2\ra^{-s}\la\xi_1\ra^{-s}}
{\la\tau-\phi(\xi)\ra^{-b'}\la\tau+\tau_1-\tau_2-\phi(\xi+\xi_1-\xi_2)\ra^b \la\tau_1-\phi(\xi_1)\ra^b\la\tau_2-\phi(\xi_2)\ra^b},
\end{equation}
then
\begin{equation}\label{K-bound}
I(\xi,\tau)
:=
\|K(\eta,\eta_1,\eta_2)\|^2_{L^2_{\eta_1,\eta_2}}
\le
C(s,b,b')<\infty,
\end{equation}
where $\phi(\xi) = \alpha\xi^2+\beta\xi^3$ and
 $C(s,b,b')$ is a positive constant independent of $\xi$ and $\tau$.
\end{lemma}

We will use lemmas \ref{Corollary.6.3} and \ref{Lemma2.2}  to derive the following  trilinear estimates in the Gevrey-Bourgain's spaces associated to the mKdV equation and tNLS equation, respectively.
\begin{proposition}
\label{mKdV-trilinear-estimate-thm}
Let $\sigma\ge 0$, there is $\frac12<b<1$ such that
\begin{equation}\label{mKdV-trilinear-estimate}
\| \partial_x(u_1 u_2 u_3)\|_{Y^{\sigma,\frac14,b-1}}
\le
C\|u_1\|_{Y^{\sigma,\frac14,b}}\|u_2\|_{Y^{\sigma,\frac14,b}}\|u_3\|_{Y^{\sigma,\frac14,b}},
\end{equation}
where $C>0$ depends only on $b$.
\end{proposition}
\begin{proof}
The proof is done by applying the trivial inequality $e^{\sigma |\xi|}\le e^{\sigma |\xi-\xi_1-\xi_2|} e^{\sigma |\xi_1|}e^{\sigma |\xi_2|}$ and estimate \eqref{trilinear} for $e^{\sigma |D_x|} u_i$ in place of $u_i$ with $i=1,2,3$, where $e^{\sigma|D_x|}$ is the operator given in \eqref{op.A} (see Corollary $1$ in \cite{BFH-gB} for a more detailed proof).
\end{proof}

\begin{proposition}\label{a-trilinear-estimate-thm}
Let $\sigma\ge 0$, $-\frac14<s\le 0$, $b>\frac7{12}$ and $b'<\frac{s}3$, then we have
\begin{equation}\label{a-trilinear-estimate}
\| v_1v_2\overline{v}_3\|_{Z^{\sigma,s,b'}}
\le
C
\| v_1\|_{Z^{\sigma,s,b}} \| v_2\|_{Z^{\sigma,s,b}}\| v_3\|_{Z^{\sigma,s,b}}.
\end{equation}
\end{proposition}
\begin{proof}
As in the proof of Proposition \ref{mKdV-trilinear-estimate-thm}, we use the inequality $e^{\sigma |\xi|}\le e^{\sigma |\xi-\xi_1-\xi_2|} e^{\sigma |\xi_1|}e^{\sigma |\xi_2|}$, to obtain
\begin{equation}\label{B-1}
|e^{\sigma |\xi|}\widehat{v_1v_2\overline{v}_3}(\xi,\tau)|
\le
\int_{\rr^4}
|\widehat{V_1}(\xi+\xi_1-\xi_2, \tau+\tau_1-\tau_2)\widehat{V_2}(\xi_2,\tau_2)\overline{\widehat{V_3}}(\xi_1,\tau_1)| d\xi_2d\tau_2d\xi_1d\tau_1,
\end{equation}
where $V_j=e^{\sigma|D_x|}v_i$ for $j=1,2,3$ and we considered a change of variables $(\xi_1,\tau_1)\rightarrow (-\xi_1,-\tau_1)$.

To simplify the exposition, let us define
\begin{equation}
\label{Carvajal.notation}
f(\xi,\tau)
=
\la\xi\ra^s\la\tau-\phi(\xi)\ra^b|\widehat{V_1}|, \;\;
g(\xi,\tau)
=
\la\xi\ra^s\la\tau-\phi(\xi)\ra^b|\widehat{V_2}|, \;\;
h(\xi,\tau)
=
\la\xi\ra^s\la\tau-\phi(\xi)\ra^b|\widehat{V_3}|,
\end{equation}
so that $ \|f\|_{L^2}=\|v_1\|_{Z^{\sigma,s,b}}$, $ \|g\|_{L^2}= \|v_2\|_{Z^{\sigma,s,b}}$ and $ \|h\|_{L^2}= \|v_3\|_{Z^{\sigma,s,b}}$.
Also, we define
\begin{equation}\label{notation2}
\eta=(\xi,\tau), \quad \eta_1=(\xi_1,\tau_1),\;
\quad
\eta_2=(\xi_2,\tau_2).
\end{equation}

Now, using these notations, the definition of the $Z^{\sigma,s,b'}$-norm given in \eqref{G-B-norm} with $\phi(\xi) = \alpha\xi^2+\beta\xi^3$ and the estimate \eqref{B-1}, one can easily obtain
\begin{equation}\label{B-2}
\| v_1v_2\overline{v_3}\|_{Z^{\sigma,s,b'}}
\le
\Big\|\int_{\rr^4}
f(\eta+\eta_1-\eta_2)g(\eta_2)\overline{h}(\eta_1)K(\eta,\eta_1,\eta_2)d \eta_1d \eta_2
\Big\|_{L^2_\eta},
\end{equation}
where 
$K(\eta,\eta_1,\eta_2)$
is as in \eqref{Knn}.

Applying Minkowski inequality for integrals and H\"older's inequality, it is easy to obtain
\begin{equation}
\label{B-4}\!\!\!
\Bigg\| \!
\int_{\rr^4}\!\! f(\eta+\eta_1-\eta_2)g(\eta_2)\overline{h}(\eta_1)K(\eta,\eta_1,\eta_2)d{\eta_1}d{\eta_2} 
\Bigg\|_{L^2_{\eta}}\!\!\!
\le\!
\|f\|_{L^2}\|g\|_{L^2}\|h\|_{L^2}\|K(\eta,\eta_1,\eta_2)\|_{L^\infty_{\eta} L^2_{\eta_1,\eta_2}}.
\end{equation}
Using \eqref{K-bound} from  Lemma \ref{Lemma2.2}, we guarantee that 
\begin{equation}
\label{B-5}
\|K(\eta,\eta_1,\eta_2)\|_{L^\infty_\eta L^2_{\eta_1,\eta_2}}\le C(s,b,b')<\infty.
\end{equation}
Now, combining \eqref{B-2}, \eqref{B-4} and \eqref{B-5}, we arrive at
\begin{equation*}
\| v_1v_2\overline{v}_3\|_{Z^{\sigma,s,b'}}
\le
C(s,b,b')\|f\|_{L^2}\|g\|_{L^2}\|h\|_{L^2}
=
C\|v_1\|_{Z^{\sigma,s,b}} \|v_2\|_{Z^{\sigma,s,b}}\|v_3\|_{Z^{\sigma,s,b}},
\end{equation*}
finishing the proof.
\end{proof}

In what follows, we present some results that will be useful to derive  almost conserved quantities.

\begin{lemma}[Lemma $5$ in \cite{SS}]
\label{est-XT}
Let $\sigma\geq 0$, $s\in\mathbb{R}$, $-\frac12<b<\frac12$ and $T>0$. 
Then, for any time interval $I\subset [-T,T]$, we have
$$
\lm \chi_{I}w\rm_{X^{\sigma,s,b}}\leq C\lm w\rm_{X_{T}^{\sigma,s,b}},
$$
where $\chi_{I}$ is the characteristic function of $I$ and $C>0$ depends only on $b$.
\end{lemma}

\begin{lemma}[Lemma $3$ in \cite{BFH}]
For $\sigma>0$, $\theta\in [0,1]$ and $\alpha,\beta,\gamma\in\mathbb{R}$, the following estimate holds
\begin{equation}\label{exp-est}
e^{\sigma|\alpha|}e^{\sigma |\beta|}e^{\sigma |\gamma|}
-
e^{\sigma|\alpha +\beta+\gamma|}
\leq
\left[
2\sigma \min\left\{ |\alpha|+ |\beta|, |\alpha|+ |\gamma|, |\beta|+ |\gamma| \right\}
\right]^{\theta} 
e^{\sigma|\alpha|}e^{\sigma|\beta|}e^{\sigma|\gamma|}.
\end{equation}
\end{lemma}

We finish this section stating one more lemma that will be useful in obtaining almost conserved quantity for the mKdV equation, which is an immediate consequence of the well known Strichartz's type estimates
\begin{equation}\label{str}
\|u\|_{L^6_xL^6_t}
\le
C\|u\|_{Y^{0,b}}, \;\; \text{for all } b>\frac12.
\end{equation}

The following result is immediate using \eqref{str} and the generalized H\"older inequality.
\begin{lemma}
\label{Strichartz}
For all $u_1$, $u_2$, $u_3$ and $b> \frac12$, we have
\begin{equation}
\|u_1 u_2 u_3\|_{L^2_x L^2_t}
\le C
\|u_1\|_{Y^{0,b}}\|u_2\|_{Y^{0,b}}\|u_3\|_{Y^{0,b}}.
\end{equation}
\end{lemma}

%
%
%
%
%
%

\section{ Local well-posedness - Proof of theorems \ref{lwp-mKdV-thm} and \ref{lwp-tNLS-thm}}
\label{Sec-3}
In this section we prove the  local well-posedness results by using the Gevrey-Bourgain's spaces. We use the standard strategy  based on a fixed point argument for the iteration map defined via the solution of the corresponding integral equation,  commonly known as Duhamel's formula,
\begin{equation*}
w(t)
=W(t)w_0 - \int_0^{t} W(t-t')f(w)(x,t') d t',
\end{equation*}
where $W(t)$ is the semigroup associated to the linear problem, $w_0$ is the initial data and  $f(w)$ is the nonlinear part.
 This is a classical strategy found in the literature (see for example \cite{B} and \cite{KPV} for the KdV equation). For the sake of completeness we provide detailed proof for the tNLS equation, the proof for the mKdV equation follows similarly.

\begin{proof}[Proof of Theorem \ref{lwp-tNLS-thm}] 
Let  $\sigma>0$, $v_0\in G^{\sigma,s}(\rr)$ with $-\frac14< s \le 0$ and $\psi_T$ be the cut-off function as defined earlier. Let us define a  solution map $\Phi_T$ for $0<T\leq 1$ given by 
\begin{equation}\label{solution_map}
\Phi_T(v)
=
\psi(t)W(t)v_0 - \psi_T(t)\int_0^{t} W(t-t')(i\gamma |v|^2v)(x,t') d t'.
\end{equation}

Our main goal is to show the existence of a lifespan $T>0$ such that $\Phi_T$ is a contraction map on a suitable complete space. In other words, we will prove that there are $b>\frac12$ and $T_0=T_0(\|v_0\|_{G^{\sigma, s}})>0$ such that 
$
\Phi_{T_0}
:
B(r)\rightarrow B(r)
$
is a contraction map, where
$
B(r)
=
\left\{v\in X_{T_0}^{\sigma,s,b};\; \lm v\rm_{Z_{T_0}^{\sigma,s,b}}\le r\right\}
$
with
$r=2c\|v_0\|_{G^{\sigma,s}}$
and $c$ is a positive constant depending only on $s$ and $b$.

Indeed, applying the nonlinear estimate \eqref{a-trilinear-estimate}  and the linear inequalities \eqref{psiT.w.u-est} and \eqref{psiT.int.w.u-est.0}, for all $v\in B(r)$, we obtain 
\begin{equation}\label{C-2}
\| \Phi_{T_0}(v)\|_{Z^{\sigma,s,b}}
\le
c\|v_0\|_{G^{\sigma,s}} +cT_0^{\frac1a}\|v\|_{Z^{\sigma,s,b}}^3
\le
\frac r2 +cT_0^{\frac1a}r^3,
\end{equation}
where $\frac1a=1-(b-b')>0$ with $b$ and $b'$ given as in Proposition \ref{a-trilinear-estimate-thm}.

By choosing 
\begin{equation}
\label{lifespan.1}
T_0\le (2cr^2)^{-a},
\end{equation}
one can readily get from \eqref{C-2} that $ \|\Phi_{T_0}(v)\|_{Z^{\sigma,s,b}}\le r$ for all $v\in B(r)$, showing the inclusion $\Phi_{T_0}(B(r))\subset B(r)$.

Using  \eqref{psiT.int.w.u-est.0} once again, for all $v_1,v_2\in B(r)$, we have
\begin{align*}
\| \Phi_{T_0}(v_1)-\Phi_{T_0}(v_2)\|_{Z^{\sigma,s,b}}
\le
cT_0^{\frac1a}\big\||v_1|^2v_1-|v_2|^2v_2\big\|_{Z^{\sigma,s,b'}}.
\end{align*}

Since $|v_1|^2v_1-|v_2|^2v_2 = (v_1-v_2)(|v_1|^2+\overline{v}_1v_2)+\overline{(v_1-v_2)}v_2^2$, we get
\begin{align*}
\| \Phi_{T_0}(v_1)-\Phi_{T_0}(v_2)\|_{Z^{\sigma,s,b}}
&\le
cT_0^{\frac1a}\|v_1-v_2\|_{Z^{\sigma,s,b}}(\|v_1\|^2_{Z^{\sigma,s,b}}+\|v_1\|_{Z^{\sigma,s,b}}\|v_2\|_{Z^{\sigma,s,b}}+\|v_2\|^2_{Z^{\sigma,s,b}})\\
&\le
cT_0^{\frac1a}3r^2\|v_1-v_2\|_{Z^{\sigma,s,b}}.
\end{align*}
Now, if we choose $T_0$ also satisfying
\begin{equation}
\label{lifespan.2}
T_0 <
(3cr^2)^{-a},
\end{equation}
it can easily be shown that $\Phi_{T_0}$ is a contraction map.
To complete the proof, it is sufficient to choose a lifespan $0<T_0\le 1$ satisfying \eqref{lifespan.1} and \eqref{lifespan.2}. 

Therefore, $\Phi_{T_0}$ satisfies the desired requirements by considering
\begin{equation}
\label{lifespan}
T_0
=
\frac{c_0}{(1+\|v_0\|^2_{G^{\sigma,s}})^a},
\end{equation}
for an appropriate constant $c_0>0$ depending on $s$ and $b$.
Hence, $\Phi_{T_0}$ admits a unique fixed point, which is a local in time solution of \eqref{tNLS-IVP} satisfying 
\begin{equation}
\label{bound.sol-tNLS}
\|v\|_{Z^{\sigma,s,b}_{T_0}}
\le
r
=
c\|v_0\|_{G^{\sigma,s}}.
\end{equation}

Also, we can prove that $\Phi_{T_0}(v)$ depends continuously on $v_0$ in an analogous manner, thereby completing the proof of Theorem \ref{lwp-tNLS-thm}.
\end{proof}

\begin{proof}[Proof of Theorem \ref{lwp-mKdV-thm}] Idea of proof of this theorem is similar to that of Theorem \ref{lwp-tNLS-thm}. The only difference is that in this case we use the trilinear estimate \eqref{trilinear} from Lemma \ref{Corollary.6.3}. So, we omit the details.
\end{proof}

\begin{remark}
The bound of the local solution given in \eqref{bound.sol-tNLS} plays an important role in the construction of a global solution to be shown in the last section of this work. Of course, for the local solution $u\in Y_{T_0}^{\sigma, s, b}$ of the mKdV equation we also have an analogous bound
\begin{equation}
\label{bound.sol-mKdV}
\|u\|_{Y^{\sigma,s,b}_{T_0}}
\le
c\|u_0\|_{G^{\sigma,s}}.
\end{equation}
\end{remark}

%
%
%
%
%
%

\section{Almost conserved quantities} \label{Sec-4}

In this section we will introduce almost conserved quantities associated to the mKdV and tNLS equations and find appropriate estimates for them. Taking in consideration the conserved quantities in \eqref{mass-mkdv}, \eqref{ener-mkdv} and \eqref{mass-tnls},  we define for the mKdV equation
\begin{equation}
\label{E-mKdV}
E_\sigma(t)
=
\|u(t)\|^2_{G^{\sigma,1}}-\frac\mu 6 \|e^{\sigma |D_x|}u\|_{L^4_x}^4, 
\end{equation}
and for the tNLS equation
\begin{equation}
\label{M-tNLS}
M_\sigma(t)
=
\|v(t)\|^2_{G^{\sigma,0}}.
\end{equation} 

Note that, for $\sigma =0$  \eqref{E-mKdV} and \eqref{M-tNLS} turn out to be the conserved quantities \eqref{ener-mkdv} and \eqref{mass-tnls} respectively. However, for $\sigma>0$ they are no more conserved by the flow. We will show that these quantities are almost conserved by deriving some appropriate estimates. 
For this purpose, we need estimates in the Bourgain's space norm  for the following expressions 
\begin{equation}
\label{fU-def-mKdV}
F(U)
:=
\frac {\mu}3 \partial_x\Big[
U^3-e^{\sigma |D_x|}\big((e^{-\sigma|D_x|} U)^3\big)
\Big],  \text{ for the mKdV equation, }
\end{equation}
with $\mu=\pm 1$ and
\begin{equation}
\label{gU-def-tNLS}
G(V)
:=
-\Big[
|V|^2V-e^{\sigma |D_x|}\big(|e^{-\sigma|D_x|} V|^2e^{-\sigma|D_x|} V
\big)
\Big], \text{ for the tNLS equation, }
\end{equation} 
which is the content of the next two lemmas.

\begin{lemma}
\label{est.fU-mKdV}
Let $F$ be as defined in \eqref{fU-def-mKdV} and $\sigma>0$.
Then, there is some $\frac12<b<1$ such that for all $\ell\in[0,\frac34]$
\begin{align}
\label{f(U)-L2}
\|F(U)\|_{L^2_xL^2_t}
&\le
C \sigma^{\ell}\|U\|^3_{Y^{1,b}},\\
\label{f(U)-Bourgain}
\|\partial_xF(U)\|_{Y^{0,b-1}}
&\le
C \sigma^{\ell}\|U\|^3_{Y^{1,b}},
\end{align}
for some constant $C>0$ independent of $\sigma$.
\end{lemma}
\begin{proof}
We start the proof by observing that
\begin{equation}
\label{fourier-f(U)}
|\widehat{F(U)}(\xi,\tau)|
\le
C|\xi|
\int_{\ast}
(1-e^{-\sigma(|\xi_1|+|\xi_2|+|\xi_3|-|\xi|)})| \widehat{U}(\xi_1,\tau_1)|| \widehat{U}(\xi_2,\tau_2)||\widehat{U}(\xi_3,\tau_3)|,
\end{equation}
where $\int_\ast$ denotes the integral over the set $\xi=\xi_1 +\xi_2+\xi_3$ and $\tau=\tau_1+\tau_2+\tau_3$. 

Now, from the classical inequality 
\begin{equation*}
\label{exp.ineq}
e^x-1
\le
x^{\ell}e^x, 
\;\text{ for all }\; x\ge 0 \;\text{ and }\; \ell\in[0,1],
\end{equation*}
we get
\begin{equation}\label{exp-ell}
1-e^{-\sigma(|\xi_1|+|\xi_2|+|\xi_3|-|\xi|)}
\le
\sigma^{\ell} (|\xi_1|+|\xi_2|+|\xi_3|-|\xi|)^{\ell}.
\end{equation}
Let $\xi_{\max}, \xi_{\text{med}}$ and $\xi_{\min}$ be the maximum, medium and minimum values of $\{|\xi_1|,|\xi_2|,|\xi_3|\}$. As shown in \cite{T} (more precisely, see the proof of Lemma $7$ there), we have 
$$
|\xi_1|+|\xi_2|+|\xi_3|-|\xi|
\le
12\xi_{\text{med}},
$$
and consequently the estimate \eqref{exp-ell} yields
\begin{equation} \label{exp-2}
1-e^{-\sigma(|\xi_1|+|\xi_2|+|\xi_3|-|\xi|)}
\le
C\sigma^{\ell}\xi_{\text{med}}^{\ell}.
\end{equation}
Thus, using \eqref{exp-2} in \eqref{fourier-f(U)}, we obtain
\begin{equation}
\label{fourier-f(U)-est}
|\widehat{F(U)}(\xi,\tau)|
\le
C\sigma^{\ell} |\xi|
\int_{\ast}
 \xi_{\text{med}}^{\ell}| \widehat{U}(\xi_1,\tau_1)|| \widehat{U}(\xi_2,\tau_2)||\widehat{U}(\xi_3,\tau_3)|.
\end{equation}

Now, we move to prove \eqref{f(U)-L2}. 
In order to simplify the exposition, without loss of generality, we can consider $|\xi_1|\le |\xi_2|\le  |\xi_3|$.
With this consideration, one has
$
|\xi|\xi_\text{med}^{\ell}
\le 3|\xi_3||\xi_2|^{\ell},
$
and consequently from \eqref{fourier-f(U)-est} and Plancherel's identity, we obtain
\begin{equation*}
\| F(U)\|_{L^2_xL^2_t}
\le
C\sigma^{\ell}\Big\| 
\int_\ast  |\widehat{U}(\xi_1,\tau_1)|| \widehat{D_x^{\ell}U}(\xi_2,\tau_2)||\widehat{D_xU}(\xi_3,\tau_3)|\Big\|_{L^2_\xi L^2_\tau}
=
C\sigma^{\ell}
\big\| w_1 w_2 w_3\big\|_{L^2_x L^2_t},
\end{equation*}
where $w_1$, $w_2$ and $w_3$ are defined by
$\widehat{w_1}(\xi,\tau)=|\widehat{U}(\xi,\tau)|$, 
$\widehat{w_2}(\xi,\tau)=|\widehat{D_x^{\ell} U}(\xi,\tau)|$ 
 and
$\widehat{w_3}(\xi,\tau)=|\widehat{D_x U}(\xi,\tau)|$.
Then, by using Lemma \ref{Strichartz}, we obtain
\begin{align*}
\|F(U)\|_{L^2_xL^2_t}
\le
C\sigma^{\ell}
\|w_1\|_{Y^{0,b}}\|w_2\|_{Y^{0,b}} \|w_3\|_{Y^{0,b}}
\le 
C\sigma^{\ell}
\|U\|_{Y^{1,b}}^3,
\end{align*}
since $0\le \ell \le 1$, and this finishes the proof of \eqref{f(U)-L2}.

Concerning \eqref{f(U)-Bourgain}, first we observe that for $0\leq k\leq 1$, one has
\begin{equation}\label{FU-sb}
\|\partial_xF(U)\|_{Y^{0,b-1}}
\le
\big\|\la \tau-\xi^3\ra^{b-1}\la\xi\ra^{k}|\xi|^{1-k}|\widehat{F(U)}(\xi,\tau)|\big\|_{L^2_\xi L^2_\tau},
\end{equation}
since $\la\xi\ra^{-k}\le |\xi|^{-k}$, for all $\xi\neq 0$.

Assuming again $|\xi_1|\le |\xi_2|\le |\xi_3|$ and using \eqref{fourier-f(U)-est}, we get from \eqref{FU-sb} that
\begin{equation}\label{FU-sb1}
\begin{split}
\!\!\|\partial_xF(U)\|_{Y^{0,b-1}}
\!&\!\le \!
C\sigma^{\ell}
\Big\|\la \tau-\xi^3\ra^{b-1}\!\!\la\xi\ra^{k} |\xi|\int_\ast
|\xi_2|^{\ell}|\xi_3|^{1-k}| \widehat{U}(\xi_1,\tau_1)|| \widehat{U}(\xi_2,\tau_2)||\widehat{U}(\xi_3,\tau_3)| \Big\|_{L^2_\xi L^2_\tau}\\
\!&\!=
C\sigma^{\ell}
\big\| \partial_x(w_1 w_2 w_4)\big\|_{Y^{k,b-1}},
\end{split}
\end{equation}
where  in the first inequality we used $|\xi|^{1-k}\le 3^{1-k} |\xi_3|^{1-k}$ and 
$\widehat{w_4
}(\xi,\tau)=|\widehat{D_x^{1-k} U}(\xi,\tau)|$.

Now, considering $k=\frac14$, we can use the trilinear estimate \eqref{trilinear} with $b=\frac12+\varepsilon$ in \eqref{FU-sb1}, to obtain
\begin{equation}\label{FU-f}
\|\partial_xF(U)\|_{Y^{0,b-1}}
\le
C\sigma^{\ell}
\| w_1\|_{Y^{\frac 14,b}} \| w_2\|_{Y^{\frac 14,b}}\| w_4\|_{Y^{\frac 14,b}}.
\end{equation}

Finally, since $0\le \ell \le \frac34$ the estimate \eqref{FU-f} yields
\begin{equation*}\label{FU-ff}
\|\partial_xF(U)\|_{Y^{0,b-1}}
\le C\sigma^{\ell}
\| U\|_{Y^{1,b}}^3,
\end{equation*}
which proves the desired estimate \eqref{f(U)-Bourgain}.
\end{proof}

\begin{remark} As can be inferred from the proof, the estimate \eqref{f(U)-L2} holds for $0\leq \ell\leq 1$. However, the estimate \eqref{f(U)-Bourgain} holds only for $0\leq \ell\leq 3/4$. This later restriction forces us to consider the maximum exponent of $\sigma$ to be $\frac34$ in the almost conserved quantity, see \eqref{ACL-mKdV} below.

\end{remark}

\begin{lemma}
\label{est.gU-tNLS}
Let $G$ be as in \eqref{gU-def-tNLS} and $\sigma>0$.
Then, for any $\theta\in [0,\frac14)$ there is some $\frac12<b<1$ such that
\begin{equation}
\label{Bourgain-nonlinear}
\|G(V)\|_{Z^{0,b-1}}
\le
C\sigma^{\theta}\|V\|^3_{Z^{0,b}},
\end{equation}
for some constant $C>0$ independent of $\sigma$.
\end{lemma}

\begin{proof}
We start by observing that 
\begin{equation}\label{A-6}
\begin{split}
\big|\widehat{G(V)}(\xi,\tau)\big|&\le C
\int\! |e^{\sigma|\xi|}\!-\!e^{\sigma(|\xi-\xi_1-\xi_2|+|\xi_2|+|\xi_1|)}| |\widehat{v}(\xi-\xi_{1}-\xi_2,\tau-\tau_{1}-\tau_2)|\times\\
&\qquad\quad\times| \widehat{v}(\xi_2,\tau_2) ||\overline{\widehat{v}}(-\xi_1,-\tau_1)|d{\xi_{2}}d{\tau_{2}}d{\xi_{1}}d{\tau_{1}},
\end{split}
\end{equation}
where $v=e^{-\sigma |D_x| }V$.
Using the estimate \eqref{exp-est}, it follows from \eqref{A-6} that
\begin{equation}\label{GU-1}
\begin{split}
\big|\widehat{G(V)}(\xi,\tau)\big|
&\le
C\sigma^\theta\int_{\rr^4} \min\{|\xi-\xi_1-\xi_2|+|\xi_1|,|\xi-\xi_1-\xi_2|+|\xi_2|,|\xi_1|+|\xi_2|\}^\theta \times\\
&\qquad\times|\widehat{V}(\xi-\xi_{1}-\xi_2,\tau-\tau_{1}-\tau_2)|
|\widehat{V}(\xi_2,\tau_2) |
|\overline{\widehat{V}}(-\xi_1,-\tau_1)|d{\xi_{2}}d{\tau_{2}}d{\xi_{1}}d{\tau_{1}}.
\end{split}
\end{equation}

Also, we have the following inequality
\begin{equation}
\label{3.4'}
\min\{|\xi-\xi_1-\xi_2|+|\xi_1|,|\xi-\xi_1-\xi_2|+|\xi_2|,|\xi_1|+|\xi_2|\}
\le
3
\frac{\la\xi-\xi_1-\xi_2\ra \la\xi_2\ra \la\xi_1\ra}{\la \xi\ra},
\end{equation}
which is obtained from inequality ($3.4$) in \cite{BFH} by considering $\xi_1+\xi_2$ in place of $\xi_1$.

Inserting \eqref{3.4'} in \eqref{GU-1}, we obtain
\begin{equation}\label{Est-G}
\begin{split}
\big|\widehat{G(V)}(\xi,\tau)\big|
&\le C\sigma^\theta
\int \la\xi+\xi_1-\xi_2\ra^\theta \la\xi_2\ra^\theta \la\xi_1\ra^{\theta} \la \xi\ra^{-\theta} |\widehat{V}(\xi+\xi_{1}-\xi_2,\tau+\tau_{1}-\tau_2)|\times\\
&\qquad\qquad\times|\widehat{V}(\xi_2,\tau_2) |
|\overline{\widehat{V}}(\xi_1,\tau_1)|d{\xi_{2}}d{\tau_{2}}d{\xi_{1}}d{\tau_{1}},
\end{split}
\end{equation}
where we made the change of variables $(\xi_1,\tau_1)\rightarrow (-\xi_1,-\tau_1)$.

Now, using the notations in \eqref{notation2},  one can easily obtain from \eqref{Est-G} that
\begin{equation}
\label{Bourgain-nonlinear.2}
\begin{split}
\|G(V)\|_{Z^{0,b-1}}
&\le
C\sigma^\theta
\Bigg\|
\int f(\eta+\eta_1-\eta_2)f(\eta_2)\overline{f}(\eta_1)K(\eta,\eta_1,\eta_2)d{\eta_1}d{\eta_2} 
\Bigg\|_{L^2_{\eta}},
\end{split}
\end{equation}
where $f(\xi,\tau)=\la \tau-\phi(\xi)\ra^b |\widehat{V}(\xi,\tau)|$ and
$$
K(\eta,\eta_1,\eta_2)
=
\frac{\la\xi\ra^{-\theta}\la\xi+\xi_1-\xi_2\ra^\theta \la\xi_2\ra^\theta \la\xi_1\ra^\theta}
{\la\tau-\phi(\xi)\ra^{1-b} \la\tau+\tau_1-\tau_2-\phi(\xi+\xi_1-\xi_2)\ra^b\la\tau_1-\phi(\xi_1)\ra^b\la\tau_2-\phi(\xi_2) \ra^b}.
$$
Next, applying the same arguments used  to obtain \eqref{B-4}  in  Proposition \ref{a-trilinear-estimate-thm}, we have
\begin{align}
\label{Carvajal-notation}
\Bigg\|
\int f(\eta+\eta_1-\eta_2)f(\eta_2)\overline{f}(\eta_1)K(\eta,\eta_1,\eta_2)d{\eta_1}d{\eta_2} 
\Bigg\|_{L^2_{\eta}}
\le
\|f\|^3_{L^2}\|K(\eta,\eta_1,\eta_2)\|_{L^\infty_\eta L^2_{\eta_1,\eta_2}}.
\end{align}

Using Lemma \ref{Lemma2.2} with $s=-\theta$, $\frac7{12}<b<\frac{11}{12}$ and $b'=b-1$, 
it follows from \eqref{Bourgain-nonlinear.2} and \eqref{Carvajal-notation} the following estimate
\begin{align*}
\|G(V)\|_{Z^{0,b-1}}
&\le
C\sigma^\theta\|f\|_{L^2}^3
=
C\sigma^\theta\|V\|^3_{Z^{0,b}},
\end{align*}
thereby finishing the proof of \eqref{Bourgain-nonlinear}.
\end{proof}

In sequel we  use the estimates obtained in lemmas \ref{est.fU-mKdV} and \ref{est.gU-tNLS} to prove that the quantities $E_{\sigma}(t)$ and $M_{\sigma}(t)$ defined in \eqref{E-mKdV} and \eqref{M-tNLS} are almost conserved. This will be the content of the following propositions.

\begin{proposition}
\label{ACL-mKdV-thm}
Let $\sigma >0$ and $\ell\in[0,\frac34]$. There exist $C>0$ and $b>\frac12$ such that for any solution $u\in Y^{\sigma,1,b}_T$ to the IVP \eqref{mKdV-IVP} in the  interval $[0,T]$, we have
\begin{equation}
\label{ACL-mKdV}
\sup\limits_{t\in [0,T]} E_\sigma(t)
\le
E_\sigma(0) + C\sigma^\ell \|u\|^4_{Y^{\sigma,1,b}_T} \big(1+\|u\|^2_{Y^{\sigma,1,b}_T}\big),
\end{equation}
where $E_\sigma(t)$ is defined in \eqref{E-mKdV}.
\end{proposition}

\begin{proof}
Let $U=e^{\sigma |D_x|}u$. First, we observe that
\begin{equation}
\label{dt-A.1}
\frac{d}{dt}\big(E_\sigma(t)\big)
=
2\int U\p_t Ud x +2\int \p_xU\p_x(\p_t U)d x -\frac {2\mu}{3} \int U^3\p_tU d x.
\end{equation}
Applying the operator $e^{\sigma |D_x|}$ to the mKdV equation \eqref{mKdV-IVP}, we get
\begin{equation}
\label{mKdV-analytic}
\p_tU+\p_x^3U+\mu U^2\p_xU
=
F(U),
\end{equation}
where $F(U)$ is defined as in \eqref{fU-def-mKdV}. Using \eqref{mKdV-analytic} in each term of \eqref{dt-A.1}, we obtain
\begin{align*}
\int U\p_t Ud x
&= -\int U\p_x^3Ud x -\frac \mu4 \int \p_x(U^4)d x +\int UF(U)d x, \\
\int \p_xU\p_x(\p_t U)d x
&= -\int \p_xU\p_x^4Ud x -\mu\int \p_xU\p_x(U^2 \p_xU)d x +\int \p_xU\p_x(F(U))d x,\\
 \int U^3\p_tU d x
&=-\int U^3\p_x^3U d x -\frac \mu6\int \p_x(U^6)d x +\int U^3F(U)d x .
\end{align*}
It follows from integration by parts and the fact that $U$ and all its spatial derivatives tend to zero as $|x|$ tends to infinity (see \cite{SS} for a detailed argument) that
\begin{align}
\label{intU-1}
\int U\p_t U d x
&= \int UF(U)d x, \\
\label{intU-2}
\int \p_xU\p_x(\p_t U)d x
&= -\frac \mu 3\int U^3 \p_x^3Ud x +\int \p_xU\p_x(F(U))d x,\\
\label{intU-3}
 \int U^3\p_tU d x
&=-\int U^3\p_x^3U d x +\int U^3F(U)d x .
\end{align}

Now, plugging   \eqref{intU-1}, \eqref{intU-2} and \eqref{intU-3} in \eqref{dt-A.1}, we arrive at
\begin{align}
\label{dt-A}
\frac{d}{dt}E_\sigma(t)
&=
2\int UF(U)d x +2\int \p_xU\p_x(F(U))d x -\frac {2\mu}{3} \int U^3F(U) d x.
\end{align}

Integrating \eqref{dt-A} in time over $[0,t']$ for $0<t'\le T$, we obtain
\begin{align}
\label{A-R}
E_\sigma(t')= E_\sigma(0) + R_\sigma(t'),
\end{align}
where
$$
R_\sigma(t')
=
2\iint \chi_{[0,t']}UF(U)d x d t 
+
2\iint \chi_{[0,t']}\p_xU\p_x(F(U))d x d t
-
\frac {2\mu}{3} \iint \chi_{[0,t']}U^3F(U) d x d t.
$$

Now, we move to estimate $|R_\sigma(t')|$ for all $0<t'\le T$.
For the first and the third term of $R_\sigma(t')$ we use Cauchy-Schwarz inequality, lemmas \ref{Strichartz} and \ref{est-XT} and estimate \eqref{f(U)-L2} restricted to time, to obtain
\begin{equation}
\label{R-1}
\Big|\iint \chi_{[0,t']}UF(U)d x d t \Big|
\le
\|\chi_{[0,t']}U\|_{L^2_xL^2_t}
\|\chi_{[0,t']}F(U)\|_{L^2_xL^2_t}
\le
C\sigma^{\ell}\|u\|_{Y^{\sigma,1,b}_T}^4
\end{equation}
and
\begin{equation}
\label{R-3}
\Big| \iint \chi_{[0,t']}U^3F(U) d x d t\Big|
\le
\|\chi_{[0,t']}U^3\|_{L^2_xL^2_t}
\|\chi_{[0,t']}F(U)\|_{L^2_xL^2_t}
\le
C\sigma^{\ell}\|u\|_{Y^{\sigma,1,b}_T}^6,
\end{equation}
for all $0<t'\le T$.

On the other hand, we apply Cauchy-Schwarz inequality, estimate \eqref{f(U)-Bourgain} restricted to time and Lemma \ref{est-XT} to have the following estimate for the second term of $R_\sigma(t')$
\begin{equation}
\label{R-2}
\Big|\iint \chi_{[0,t']}\p_xU\p_xF(U)d x d t\Big|
\le
\|\chi_{[0,t']}\p_xU\|_{Y^{0,1-b}}
\|\chi_{[0,t']}\p_xF(U)\|_{Y^{0,b-1}}
\le
C\sigma^{\ell}\|u\|_{Y^{\sigma,1,b}_T}^4,
\end{equation}
 for some $\frac12<b<1$.
 
 Finally, using  \eqref{R-1}, \eqref{R-3} and \eqref{R-2} in  \eqref{A-R} the required estimate \eqref{ACL-mKdV} follows.
\end{proof}

\begin{corollary}\label{Cor-4}
Let $\sigma >0$,  $\ell\in[0,3/4]$ and $E_\sigma(t)$ as defined in \eqref{E-mKdV}. There exists $C>0$ such that for any solution $u\in Y^{\sigma,1,b}_T$ to the IVP \eqref{mKdV-IVP} in the defocusing case $(\mu = -1)$, we have
\begin{equation}
\label{ACL-mKdV-bound}
\sup\limits_{t\in [0,T]} E_\sigma(t)
\le
E_\sigma(0) + C\sigma^{\ell} E_{\sigma}(0)^2 \big(1+E_{\sigma}(0)\big), \qquad \ell\in\Big[0, \frac34\Big].
\end{equation}

\end{corollary}
\begin{proof}
First note that, for $\mu=-1$ from \eqref{E-mKdV}, we have
\begin{equation} \label{defocusing}
E_{\sigma}(0) = \|u_0\|^2_{G^{\sigma,1}}+\frac 16 \|e^{\sigma |D_x|} u_0 \|_{L^4_x}^4
\ge
\|u_0\|^2_{G^{\sigma,1}}.
\end{equation}

Now, using the estimates \eqref{bound.sol-mKdV} and \eqref{defocusing} in  the almost conserved quantity \eqref{ACL-tNLS}, we get the required estimate \eqref{ACL-mKdV-bound}.
\end{proof}

\begin{remark} Observe that,  for the solution to the IVP \eqref{mKdV-IVP} in the focusing case $(\mu = 1)$, from  \eqref{E-mKdV} we obtain
\begin{equation}
\label{focusing}
E_{\sigma}(0)
=
\|u_0\|^2_{G^{\sigma,1}}-\frac 16 \|e^{\sigma |D_x|} u_0 \|_{L^4_x}^4
\le
\|u_0\|^2_{G^{\sigma,1}},
\end{equation}
which cannot be used to obtain an estimate of the the form \eqref{ACL-mKdV-bound}. As can be seen in the proof of Theorem \ref{global-mKdV-thm} the estimate  \eqref{ACL-mKdV-bound} plays a crucial role in our argument. For this reason, we only consider the defocusing mKdV equation to obtain the lower bound for the evolution of the radius of analyticity.
\end{remark}

\begin{proposition}
\label{ACL-tNLS-thm}
Let $\sigma >0$ and $\theta\in[0,\frac14)$. There exists $C>0$ and $\frac12<b<1$ such that for any solution $v\in Z^{\sigma,0,b}_T$ to the IVP \eqref{tNLS-IVP} in the  interval $[0,T]$, we have 
\begin{equation}
\label{ACL-tNLS}
\sup\limits_{t\in [0,T]} M_\sigma(t)
\le
M_\sigma(0) + C\sigma^\theta \|v\|^4_{Z^{\sigma,0,b}_T},
\end{equation}
where $M_\sigma(t)$ is defined as in \eqref{M-tNLS}.
\end{proposition}

\begin{proof}
Applying $e^{\sigma |D_x|}$ to the tNLS equation in \eqref{tNLS-IVP} and denoting $V=e^{\sigma |D_x|}v$, we obtain
\begin{equation}\label{HONLS-U}
\partial_t V +i\alpha \partial_x^2V +\beta\partial_x^3 V + i\gamma |V|^2V
=
-i\gamma G(V),
\end{equation} 
where $G(V)$ is defined as in \eqref{gU-def-tNLS}.

Now, multiplying \eqref{HONLS-U} by $\overline{V}$ and considering the real parts, we get
\begin{equation}
\label{re.HONLS-U}
\text{Re}(\overline{V}\partial_tV)-\alpha\im(\overline{V}\partial_x^2V) +\beta\re(\overline{V}\partial_x^3V)
=
\gamma\im\big(\overline{V}G(V)\big),
\end{equation}
since  $\alpha, \beta$ and $\gamma$ are real constants. One can infer from \eqref{re.HONLS-U} that
\begin{equation}
\label{HONLS-U-2}
\frac 12 \partial_t(|V|^2) -\alpha\im(\partial_x(\overline{V}\partial_xV))+\beta\re(\overline{U}\partial_x^3V)
=
\gamma\im\big(\overline{V}G(V)\big).
\end{equation}

Integrating \eqref{HONLS-U-2} with respect to the space variable, we get
\begin{equation}
\label{HONLS-U-int_x}
\small{
\frac 12\frac{d}{dt}\!\int\! |V|^2d x
-
\alpha\!\int\! \im (\partial_x(\overline{V}\partial_xV))d x
+
\beta\!\int\! \re(\overline{V}\partial_x^3 V)d x
=
\gamma\!\int\! \im\big(\overline{V}G(V)\big)d x.}
\end{equation}
Since $V$ and all its spatial derivatives tend to zero when $|x|$ tends to infinity, using integration by parts, we get 
\begin{equation}\label{eq-43}
\frac{d}{dt}\!\int |V|^2d x
=
2\gamma 
\int \im\big(\overline{V} G(V)\big)d x.
\end{equation}
Now, integrating \eqref{eq-43} in time over the interval $[0,t']$ for any $0<t'\le T$, we obtain
\begin{equation}
\label{alm.conserv.0}
\| v(t')\|^2_{G^{\sigma,0}}
=
\| v(0)\|^2_{G^{\sigma, 0}}
+
2\gamma\im\Big(  \iint \chi_{[0,t']}(t)\overline{V} G(V)d x d t\Big).
\end{equation}
Using Plancherel's identity and H\"older's inequality, we estimate the integral in the right side of \eqref{alm.conserv.0} as
\begin{equation}\label{A-5}
\Big| \iint \chi_{[0,t']}(t) \overline{V} G(V)d x d t\Big|
\le
\|\chi_{[0,t']}G(V)\|_{Z^{0,b-1}}
\|\chi_{[0,t']}V\|_{Z^{0,1-b}},
\end{equation}
where $\frac12<b<1$ is chosen as in Lemma \ref{est.gU-tNLS}. 
Furthermore, using the fact that $t'<T$, Lemma \ref{est-XT} and $1-b<b$, we get
\begin{align}
\label{alm.conserv.1}
\Big| \iint \chi_{[0,t']}(t) \overline{V} G(V)d x d t\Big|
&\le
C\|v\|_{Z_{T}^{\sigma,0,b}}\|G(V)\|_{Z^{0,b-1}_{T}}.
\end{align}
Finally, putting together \eqref{alm.conserv.0}, \eqref{alm.conserv.1} and \eqref{Bourgain-nonlinear} (restricted to the interval $[0,T]$), we conclude that one can choose  $b>\frac 12$ such that \eqref{ACL-tNLS} goes true.
\end{proof}

Combining the bound \eqref{bound.sol-tNLS} with the almost conserved quantity \eqref{ACL-tNLS} the following result follows immediately.

\begin{corollary} Let $\sigma >0$ and $M_\sigma(t)$ be as defined as in \eqref{M-tNLS}. There exists $C>0$  such that for any solution $v\in Z^{\sigma,0,b}_T$ to the IVP \eqref{tNLS-IVP} in the  interval $[0,T]$, we have
\begin{equation}
\label{ACL-tNLS-bound}
\sup\limits_{t\in [0,T]} M_\sigma(t)
\le
M_\sigma(0) + C\sigma^\theta M_{\sigma}(0)^2, \qquad \theta\in\Big[0, \frac14\Big).
\end{equation}
\end{corollary}

%
%
%
%
%
%
\section{Global analytic solution - Proof of theorems \ref{global-mKdV-thm} and \ref{global-tNLS-thm}}\label{Sec-5}

We start by observing that if we prove the extension of the solution for $s=s_0$ as stated in theorems \ref{global-mKdV-thm} and \ref{global-tNLS-thm}, then it can be proved for all general $s\in \rr$  using the inclusion  \eqref{Gds.emb} (for more details we refer the works \cite{BFH}, \cite{SS} and \cite{T}).

Also, due to time reversibility of the mKdV and tNLS equations, it suffices to consider $t\geq 0$. Idea of proof of theorems \ref{global-mKdV-thm} and \ref{global-tNLS-thm} is similar using the almost conserved quantities in \eqref{ACL-mKdV-bound} and \eqref{ACL-tNLS-bound}. For the sake of completeness we give details for the proof of Theorem \ref{global-mKdV-thm} and provide some hint for Theorem \ref{global-tNLS-thm}.

\begin{proof}[Proof of Theorem \ref{global-mKdV-thm}]

Let $\ell\in[0, \frac34]$ and $\mu =-1$. Taking in consideration the above discussion, let $u_0\in G^{\sigma_0, 1}(\R)$. For initial data with other values of $s$ the proof follows by  using the inclusion  \eqref{Gds.emb} (see \cite{SS} for a detailed argument). Also, the time reversibility of the mKdV equation allows us to consider $t\geq 0$. So, for given any $T\geq T_0$, we will prove that the local solution $u$ to the IVP \eqref{mKdV-IVP} guaranteed by Theorem \ref{lwp-mKdV-thm} belongs to $C([0,T], G^{\sigma(T),1}(\rr))$, with
$
\sigma(T)
=
\min \left\{ \sigma_{0}, \dfrac{c}{T^{\frac1\ell}}\right\}.
$

From Theorem \ref{lwp-mKdV-thm} one can  infer  the existence of  a maximal lifespan
$
T^{*}:=T^{*}(\|u_0\|_{G^{\sigma_{0},1}})\in(0,\infty]
$
such that
$
u\in C([0,T^{*}), G^{\sigma_{0},1}(\rr)).
$ 
We assume that $T^{\ast}<\infty$, since otherwise we would have $T^\ast=\infty$ and the radius of analyticity would remain the same $\sigma_0$ for all time $T\geq T_0$. 
Therefore, we just need to prove the following
\begin{equation}\label{goal-T}
u\in C([0,T], G^{\sigma(T),1}(\rr)),
\;\; 
\text{for all } T\ge T^{*}.
\end{equation}

	Observe that
\begin{align*}
E_{\sigma_0}(0)
&=
\|u_0\|^2_{G^{\sigma_0,1}}
+
\frac{1}{6} \|e^{\sigma_0 |D_x|}u_0\|^4_{L^4}\\
&\le
\|u_0\|^2_{G^{\sigma_0,1}}
+
C \|D_x(e^{\sigma_0 |D_x|}u_0)\|_{L^2}\|e^{\sigma_0 |D_x|}u_0\|^3_{L^2}\\
&\le
\|u_0\|^2_{G^{\sigma_0,1}}
+
C \|u_0\|^4_{G^{\sigma_0,1}}<\infty,
\end{align*}	
where we used Gagliardo-Nirenberg inequality. Additionally, since $E_{\sigma_0}(0)\ge \|u_0\|^2_{G^{\sigma_0,1}}$ from \eqref{defocusing}, we can take the lifespan $T_0$ given as follows
$$
T_0
=
\frac{c_{0}}{\left(1+E_{\sigma_0}(0)\right)^{a}},
$$
with $c_{0}>0$, $a>1$ as in \eqref{lifespan}.
	
	The proof will be given by applying the local well-posedness result iteratively as many times as necessary to reach any given time $T>T_0$. For this purpose, we fix the time step
\begin{equation}
\label{time.step}
0<\rho
=
\frac{c_{0}}{\left(1+2E_{\sigma_0}(0)\right)^{a}}
<
T_0.
\end{equation}
 In what follows, we will describe in detail the induction steps until obtaining the desired extension of the solution.

\noindent
{\bf Extension in $[0,\rho]$.}	This is the trivial step, since from Theorem \ref{lwp-mKdV-thm},  for any $0<\sigma \le \sigma_0$, we already have $u\in C([0,\rho]; G^{\sigma, 1}(\rr))$. Furthermore, the solution $u$ satisfies
\begin{equation}
\label{CL-1}
\sup\limits_{t\in [0,\rho]} E_{\sigma}(t)
\le
E_{\sigma}(0) + C\sigma^\ell E_{\sigma}(0)^2 \big(1+E_{\sigma}(0)\big)\\
\le
E_{\sigma}(0)  + 8C\sigma^\ell E_{\sigma_0}(0)^2 \big(1+E_{\sigma_0}(0)\big).
\end{equation}
The above inequality follows from the bound \eqref{ACL-mKdV-bound}.

\noindent
{\bf Extension in $[\rho,2\rho]$.}	  If we assume 
$$
8C\sigma^\ell E_{\sigma_0}(0) \big(1+E_{\sigma_0}(0)\big) \le1,
$$
from \eqref{CL-1}, we get
\begin{equation}
\label{bound.rho}
\|u(\rho)\|^2_{G^{\sigma,1}}
\le
E_{\sigma}(\rho)
\le
\big[1+8C\sigma^\ell E_{\sigma_0}(0) \big(1+E_{\sigma_0}(0)\big)\big] E_{\sigma_0}(0)
\le 
2 E_{\sigma_0}(0),
\end{equation}
since $\sigma\le\sigma_0$. Therefore,  applying the local well-posedness result for the initial data $u(\rho)$ instead of $u_0$, we obtain that the solution $u$ belongs to $C([\rho, 2\rho]; G^{\sigma,1}(\rr))$. Additionally,  applying the almost conserved quantity \eqref{ACL-mKdV-bound} for the initial time $\rho$, the bound \eqref{bound.rho} and considering again inequality \eqref{CL-1} from the previous step, we have 
\begin{align}
\label{CL-2}
\sup\limits_{t\in [\rho,2\rho]} E_{\sigma}(t)
&\le
E_{\sigma}(\rho) + C\sigma^\ell E_{\sigma}(\rho)^2 \big(1+E_{\sigma}(\rho)\big) \nonumber\\
&\le
E_{\sigma}(\rho) + 8C\sigma^\ell E_{\sigma_0}(0)^2 \big(1+E_{\sigma_0}(0)\big) \nonumber\\
&\le \nonumber
E_{\sigma}(0) + 2\cdot 8C\sigma^\ell E_{\sigma_0}(0)^2 \big(1+E_{\sigma_0}(0)\big).
\end{align}

\noindent
{\bf Extension in $[(n-1)\rho,n\rho]$.} More generally, assuming
$$
(n-1)8C\sigma^\ell E_{\sigma_0}(0) \big(1+E_{\sigma_0}(0)\big) \le1,
$$
we can  guarantee the bound
$
E_{\sigma}((n-1)\rho)
\le 
2 E_{\sigma_0}(0)
$,
and consequently we can apply the local well-posedness result to extend the solution $u$ to the space $C([(n-1)\rho, n\rho]; G^{\sigma,1}(\rr))$.

Proceeding in this way, the induction stops at the first integer $n$ for which 
\begin{equation}
\label{final.time}
n8C\sigma^\ell E_{\sigma_0}(0) \big(1+E_{\sigma_0}(0)\big) >1,
\end{equation}
and we have reached the time $T=n\rho$ for the extension of the solution.
Now, using $T=n\rho$ in \eqref{final.time}, we obtain
\begin{equation}
\label{final.time.2}
\frac{T}{\rho} 8C\sigma^\ell E_{\sigma_0}(0) \big(1+E_{\sigma_0}(0)\big)>1.
\end{equation}
Note that $T$ can be chosen large as we want if $\sigma$ is small enough. 
Furthermore, \eqref{time.step} and  \eqref{final.time.2} imply
\begin{equation}\label{Lb}
\sigma
>
\left[\frac{\rho}{T8C\sigma^\ell E_{\sigma_0}(0) \big(1+E_{\sigma_0}(0)\big)} \right]^{\frac 1\ell}
=:
cT^{-\frac 1\ell},
\end{equation}
where $c$ depends on $c_{0}$, $\sigma_{0}$, $\ell$ and $\|u_0\|_{G^{\sigma_0, 1}} $. 
Considering $\ell=\frac34$ in \eqref{Lb}, which is the maximum that can be considered in view of Proposition \ref{ACL-mKdV-thm}, finishes the proof  for $s=1$. For other values of $s\in\R$, the proof follows using the inclusion  \eqref{Gds.emb} as described above.
\end{proof}

\begin{proof}[Proof of Theorem \ref{global-tNLS-thm}] The proof of this theorem follows using similar steps as in the proof of Theorem~\ref{global-mKdV-thm}. In this case we consider $s=0$ and use \eqref{ACL-tNLS-bound} with $\theta \in [0, \frac14)$. Finally, the proof is concluded considering the maximum value of $\theta\in[0, \frac14)$.
\end{proof}




\vskip 0.3cm
\noindent{\bf Acknowledgements.}  The authors would  like to thank the unanimous referee whose comments helped to improve the original manuscript.
The first author acknowledges the support from FAPESP, Brazil (\#2021/04999-9).
The second author acknowledges the grants from FAPESP, Brazil (\#2023/06416-6) and CNPq, Brazil (\#307790/2020-7).



\begin{thebibliography}{99}

\bibitem{Agr-07} G. Agrawal, {\em Nonlinear Fiber Optics}, Fourth Edition, Elsevier Academic Press, Oxford (2007).

\bibitem{AKS} J. Ahn, J. Kim and I. Seo, {\em On the radius of spatial analyticity for defocusing nonlinear Schrödinger equations}, Discrete and Continuous Dynamical Systems {\bf 40} (2020) 423--439.


\bibitem{BFH-gB} 
R. Barostichi, R. Figueira and A. Himonas,
{\em Well-posedness of the good Boussinesq equation in analytic Gevrey spaces and time regularity,}
J. Differential Equations 
{\bf 267} 
(2019) 
3181--3198.

\bibitem{BFH}
R. Barostichi, R. Figueira and A. Himonas 
\textit{The modified KdV equation with higher dispersion in Sobolev and analytic spaces on the line,} 
J. Evol. Equ. (2021) 2213–2237.

\bibitem{BGK} J. L. Bona, Z. Gruji\'c and H. Kalisch, 
\textit{Algebraic lower bounds for the uniform radius of spatial analyticity for the generalized KdV equation},
Ann Inst. H. Poincar\'e {\bfseries 22} (2005) 783--797.

\bibitem{BS-75} J. L. Bona, R. Smith, {\em The initial-value problem for the Korteweg-de Vries equation}, Philos. Trans. Roy.
Soc. London Ser. A {\bf 278}  (1975)  555--601.

\bibitem{B}
J. Bourgain; 
{\em Fourier transform restriction phenomena for certain lattice subsets and applications to nonlinear evolution equations. Part 2: KdV equation,} 
Geom. Funct. Anal. {\bf 3} (1993 209--262.

\bibitem{B-93} J. Bourgain, {\em Fourier transform restriction phenomena for certain lattice subsets and applications to nonlinear evolution equations I. Schr\" odinger equations}, 
Geom. Funct. Anal., {\bf 3} (2) (1993) 107--156.



\bibitem{XC-04} X. Carvajal, {\em Local well-posedness for a higher order nonlinear Schr\"odinger equation in Sobolev spaces of negative indices}, Electronic J Diff Equations {\bf 2004} (2004) 1--10.

\bibitem{CP-24} X. Carvajal, M.  Panthee, {\em Sharp global well-posedness for the cubic nonlinear Schr\"odinger equation with third order dispersion}, Journal of Fourier Analysis and Applications (2024) 30:25.

\bibitem{CP-22} X. Carvajal, M.  Panthee, {\em Nonlinear Schr\" odinger equations with the third order dispersion on modulation spaces}, Partial Differ. Equ. Appl. {\bf 3} (2022) Paper No. 59, 21 pp.



\bibitem{CKSTT} J. Colliander, M. Keel, G. Staffilani, H. Takaoka, and T. Tao, 
{ \em Sharp global well-posedness for KdV and modified KdV on $\rr$ and $\mathbb{T}$}, J. Amer. Math. Soc. {\bf 16}, No. 3, (2003), 705--749.

\bibitem{DT-21} A. Debussche, Y. Tsutsumi, {\em Quasi-invariance of Gaussian measures transported by the cubic NLS with third-order dispersion on $\T$}, J. Funct. Anal. {\bf 281} (2021),  109032, 23 pp.


\bibitem{FL} J. Forlano,  {\em A remark on the well-posedness of the modified KdV equation in $L^2$},  	arXiv:2205.13110 

\bibitem{PG} P.  Gaillard, {\em The mKdV equation and multi-parameters rational solutions}, Wave Motion {\bf 100} (2021) 102667.

\bibitem{GK-1} Z. Gruji\'c, H. Kalisch; {\em Local well-posedness of the generalized Korteweg-de Vries equation in spaces of analytic functions,} Differential and Integral Equations, {\bf 15} (2002) 1325--1334.

\bibitem{HK} 
A. Hasegawa and Y. Kodama, 
{\em Nonlinear pulse propagation in a monomode dielectric guide,}
IEEE J. Quantum Electronins {\bf 23 } (1987), 510--524.

 
\bibitem{HK-81} A. Hasegawa and Y. Kodama, {\em Signal transmission by optical solitons in monomode Fiber}, Proc. IEEE {\bf 69} (1981) 1145--1150.

\bibitem{He} M. A. Helal, {\em Soliton solution of some nonlinear partial differential equations
and its applications in fluid mechanics}, Chaos, Solitons and Fractals {\bf
13} (2002) 1917--1929.

\bibitem{H-73} R. Hirota, {\em Exact envelope-soliton solutions of a nonlinear wave equation}
Jour. Of Math. Phys.,  {\bf14}  (1973) 805--809.

\bibitem{HW} Huang, J., Wang, M.
{\em New lower bounds on the radius of spatial analyticity for the KdV equation}, J. Differ. Equations {\bf 266},  No. 9, (2019), 5278--5317.


\bibitem{Kato} T. Kato, {\em On the Cauchy problem for the (generalized) Korteweg-de Vries equation}, In Studies in applied
mathematics, volume 8 of Adv. Math. Suppl. Stud., pages 93–128. Academic Press, New York, 1983.

\bibitem{KPV}
C.E. Kenig, G. Ponce and L. Vega, 
{\em Well-posedness and scattering results for the generalized Korteweg-de Vries
equation via the contraction principle,}
Comm. on Pure and Applied Mathematics, {\bf 46} (1993) 527–620.

\bibitem{KPV-1}
C.E. Kenig, G. Ponce and L. Vega, 
{\em On the (generalized) Korteweg-de Vries equation}, Duke Math. J. {\bf 59} (1989)
 585--610.

\bibitem{K-09} N. Kishimoto, {\em Well-posedness of the Cauchy problem for the Korteweg-de Vries equation at the critical
regularity}, Differential Integral Equations {\bf 22} (2009) 447--464.

\bibitem{K} 
Y. Kodama, 
{\em Optical solitons in a monomode fiber,}
J. Stat. Phys., {\bf 39 }(1985) 597--614.


\bibitem{KEC} A. H. Khater, O. H. El-Kalaawy, D. K. Callebaut,  {\em B\"acklund transformations and exact solutions for Alfven solitons in a relativistic electronpositron plasma}, Phys. Scripta,  {\bf  58} (1998) 545--548.

\bibitem{KdV} D. J. Korteweg, G.  de Vries {\em On the change of form of long waves advancing in a rectangular canal, and on a new type of long stationary waves}, Phil. Mag. {\bf 39} (1895) 422--443.

\bibitem{L1} 
C. Laurey, 
{\em The Cauchy Problem for a Third Order Nonlinear Schr\"odinger Equation,} Nonlinear Analysis, TMA {\bf 29} (1997), 121--158.

\bibitem{LGM} H. Leblond, P. Grelu, D. Mihalache, {\em Models for supercontinuum generation
beyond the slowly-varying-envelope approximation}, Phys. Rev. A {\bf 90}
(2014) 053816-1-9.

\bibitem{LM}  H. Leblond, D. Mihalache, {\em Few-optical-cycle solitons: Modified Korteweg-de Vries sine-Gordon equation versus other non-slowly-varying-envelope approximation models}, Phys. Rev. A {\bf 79} (2009) 063835-1-7.

\bibitem{LL} Z. P. Li, Y. C. Liu, {\em Analysis of stability and density waves of traffic flow model in an ITS environment}, Eur. Phys. J. B, {\bf 53} (2006) 367--374.

\bibitem{MT-18} T. Miyaji, Y. Tsutsumi, {\em Local well-posedness of the NLS equation with third order dispersion in negative Sobolev spaces},  Differential Integral Equations {\bf 31} (2018) 111--132.

\bibitem{MT-17} T. Miyaji, Y. Tsutsumi, {\em Existence of global solutions and global attractor for the third order Lugiato-Lefever equation on $\T$}, Ann. I. H. Poincar\'e-AN {\bf 34} (2017) 1707--1725.

\bibitem{OTT-19} T. Oh, Y. Tsutsumi, N. Tzvetkov, {\em Quasi-invariant Gaussian measures for the cubic nonlinear Schr\"odinger equation with third-order dispersion}, C. R.Acad.Sci.Paris,Ser.I {\bf 357} (2019) 366--381.


\bibitem{Oikawa-93} M. Oikawa, {\em Effect of the third-order dispersion on the nonlinear Schr\" odinger equation}, J. Phys. Soc. Japan, {\bf 62} (1993) 2324--2333.


\bibitem{MPV} L. Molinet, D. Pilod, S. Vento, {\em Unconditional uniqueness for the modified Korteweg-de Vries equation
on the line}, Rev. Mat. Iberoam. {\bf 34} (2018) 1563--1608.




\bibitem{SS-18} S. Selberg; {\em Spatial analyticity of solutions to nonlinear dispersive PDE,} Nonlinear partial differential equations, mathematical physics, and stochastic analysis, 437--454, EMS Ser. Congr. Rep., Eur. Math. Soc., Zürich, 2018.


\bibitem{SS} 
S. Selberg and D. O. Silva, 
{\em Lower bounds on the radius of a spatial Analyticity for the KdV equation,}
Ann. Henri Poincar\`e, {\bf 18 }(2017) 1009--1023.


\bibitem{T-72} S. Tanaka, {\em Modified Korteweg-de Vries Equation and Scattering Theory},
Proc. Japan Acad., {\bf 48}  (1972) 466--469.

\bibitem{Tao-2} 
T. Tao, 
{\em Multilinear weighted convolution of $L^2$ functions and applications to nonlinear dispersive equations,} 
Published for the Conference Board of the Mathematical Sciences, Washington, DC; by the Amer. J.  math, {\bf 123} (2001) 839--908.


\bibitem{T} 
A. Tesfahun, 
{\em On the radius of spatial analyticity for cubic nonlinear Schrödinger equations,} J. of Differential Equations, {\bf 263} (2017) 7496--7512.

\bibitem{Tsutsumi-18} Y. Tsutsumi, {\em Well-Posedness and Smoothing Effect
for Nonlinear Dispersive Equations}, available in {\href{https://krieger.jhu.edu/math/wp-content/uploads/sites/62/2018/03/jami2018lecture2abstract1.pdf}{https://krieger.jhu.edu/math/wp-content/uploads/sites/62/2018/03/jami2018lecture2abstract1.pdf}}, 2018.



\end{thebibliography}
\end{document}